\documentclass[letter, reqno,11pt]{amsart}

\usepackage[usenames,dvipsnames]{color}
\usepackage[utf8]{inputenc}
\usepackage{amsthm,amsfonts,amssymb,amsmath,amsxtra}
\usepackage[all]{xy}
\SelectTips{cm}{}

\usepackage{nccrules}

\usepackage{verbatim}
\usepackage{xr-hyper}
\usepackage{varioref}
\usepackage{hyperref}
\usepackage[nameinlink]{cleveref}
\usepackage{tikz-cd}
\usepackage[margin=1.25in]{geometry}
\usepackage{mathrsfs}
\usepackage{xcolor}

\usepackage[shortlabels]{enumitem}
\usepackage{lipsum}
\usetikzlibrary{positioning,chains,fit,shapes,calc,arrows,patterns,external,shapes.callouts,graphs}

\usepackage[style=alphabetic]{biblatex}
\RequirePackage{xspace}
\RequirePackage{etoolbox}
\RequirePackage{varwidth}
\RequirePackage{enumitem}
\RequirePackage{tensor}
\RequirePackage{mathtools}
\RequirePackage{longtable}
\RequirePackage{multirow}

\def\ge{\geqslant}
\def\le{\leqslant}
\def\a{\alpha}

\def\G{\Gamma}

\def\D{\Delta}

\def\s{\sigma}
\def\t{\tau}

\def\k{\kappa}
\def\l{\lambda}
\def\z{\zeta}

\def\i{^{-1}}

\def\<{\langle}
\def\>{\rangle}

\newcommand{\bG}{\mathbf G}

\newcommand{{\BG}}{\ensuremath{\mathbb {G}}\xspace}

\newcommand{{\BK}}{\ensuremath{\mathbb {K}}\xspace}

\newcommand{\BQ}{\ensuremath{\mathbb {Q}}\xspace}

\newcommand{\BS}{\ensuremath{\mathbb {S}}\xspace}

\newcommand{\BZ}{\ensuremath{\mathbb {Z}}\xspace}

\newcommand{\CF}{\ensuremath{\mathcal {F}}\xspace}

\newcommand{\CI}{\ensuremath{\mathcal {I}}\xspace}
\newcommand{\CJ}{\ensuremath{\mathcal {J}}\xspace}

\newcommand{\wt}{\text{wt}}
\newcommand{\ra}{\rightarrow}

\def\tW{\tilde W}

\DeclareMathOperator{\Adm}{Adm}

\DeclareMathOperator{\Gal}{Gal}

\DeclareMathOperator{\rank}{rank}

\newcommand{\wtd}{\widetilde}

\newtheorem{theorem}{Theorem}
\newtheorem{proposition}[theorem]{Proposition}
\newtheorem{lemma}[theorem]{Lemma}

\theoremstyle{definition}
\newtheorem*{acknowledgement}{Acknowledgement}

\newtheorem{remark}[theorem]{Remark}

\numberwithin{equation}{section}
\numberwithin{theorem}{section}

\setitemize[0]{leftmargin=*,itemsep=\the\smallskipamount}
\setenumerate[0]{leftmargin=*,itemsep=\the\smallskipamount}

\renewcommand{\to}{%
   \ifbool{@display}{\longrightarrow}{\rightarrow}%
   }
\let\shortmapsto\mapsto
\renewcommand{\mapsto}{%
   \ifbool{@display}{\longmapsto}{\shortmapsto}%
   }
\newlength{\olen}
\newlength{\ulen}
\newlength{\xlen}
\newcommand{\xra}[2][]{%
   \ifbool{@display}%
      {\settowidth{\olen}{$\overset{#2}{\longrightarrow}$}%
       \settowidth{\ulen}{$\underset{#1}{\longrightarrow}$}%
       \settowidth{\xlen}{$\xrightarrow[#1]{#2}$}%
       \ifdimgreater{\olen}{\xlen}%
          {\underset{#1}{\overset{#2}{\longrightarrow}}}%
          {\ifdimgreater{\ulen}{\xlen}%
             {\underset{#1}{\overset{#2}{\longrightarrow}}}
             {\xrightarrow[#1]{#2}}}}%
      {\xrightarrow[#1]{#2}}
   }
\makeatother
\newcommand{\xyra}[2][]{%
   \settowidth{\xlen}{$\xrightarrow[#1]{#2}$}%
   \ifbool{@display}%
      {\settowidth{\olen}{$\overset{#2}{\longrightarrow}$}%
       \settowidth{\ulen}{$\underset{#1}{\longrightarrow}$}%
       \ifdimgreater{\olen}{\xlen}%
          {\mathrel{\xymatrix@M=.12ex@C=3.2ex{\ar[r]^-{#2}_-{#1} &}}}%
          {\ifdimgreater{\ulen}{\xlen}%
             {\mathrel{\xymatrix@M=.12ex@C=3.2ex{\ar[r]^-{#2}_-{#1} &}}}
             {\mathrel{\xymatrix@M=.12ex@C=\the\xlen{\ar[r]^-{#2}_-{#1} &}}}}}%
      {\mathrel{\xymatrix@M=.12ex@C=\the\xlen{\ar[r]^-{#2}_-{#1} &}}}%
   }
\makeatletter
\newcommand{\xla}[2][]{%
   \ifbool{@display}%
      {\settowidth{\olen}{$\overset{#2}{\longleftarrow}$}%
       \settowidth{\ulen}{$\underset{#1}{\longleftarrow}$}%
       \settowidth{\xlen}{$\xleftarrow[#1]{#2}$}%
       \ifdimgreater{\olen}{\xlen}%
          {\underset{#1}{\overset{#2}{\longleftarrow}}}%
          {\ifdimgreater{\ulen}{\xlen}%
             {\underset{#1}{\overset{#2}{\longleftarrow}}}
             {\xleftarrow[#1]{#2}}}}%
      {\xleftarrow[#1]{#2}}
   }
\newcommand{\isoarrow}{%
   \ifbool{@display}{\overset{\sim}{\longrightarrow}}{\xrightarrow\sim}%
   }
 
\begin{document}

\title[A dimension formula of closed ADLVs of parahoric level]{A dimension formula of closed affine Deligne-Lusztig varieties of parahoric level}

\author{Arghya Sadhukhan}
\address[A. S.]{Department of Mathematics, National University of Singapore, 10 Lower Kent Ridge Drive}

\email{arghyas0@nus.edu.sg}

\begin{abstract}
We establish a dimension formula for certain union $X^G(\mu,b)_J$ of affine Deligne-Lusztig varieties associated to arbitrary parahoric level structures of split reductive groups, under certain genericity hypotheses.
\end{abstract}

\maketitle

\vspace{-0.8cm}
\section{Introduction}

Let $F$ be a non-archimedean local field with completed maximal unramified extension of $\breve{F}$ and residue field $\kappa$. Let $(\bG,\{\mu\},b)$ be a triple where $\bG$ is a connected reductive group over $F$, $\{\mu\}$ is a conjugacy class of cocharacters $\mu: \mathbb{G}_{m,\bar{F}}\rightarrow \bG_{\bar{F}}$ and $[b] \in \bG(\breve F)$. Fix a $F$-rational parahoric level structure $J$ of $\bG$, with corresponding standard parahoric subgroup $\breve{\CJ} \subset \bG(\breve{F})$. Associated to this data, we have the following closed affine Deligne-Lusztig variety, cf. \cite{Rap05}
\begin{equation*}
    X^\bG(\mu,b)_J(\kappa):=\{g \breve \CJ \in \bG(\breve F)/\breve \CJ; g \i b \s(g) \in \breve \CJ \Adm(\mu) \breve \CJ\} \subset \CF l^\bG_J(\k)=\bG(\breve F)/\breve \CJ.
\end{equation*}

Here $X^\bG(\mu,b)_J$ is a subscheme, locally of finite type, of the partial affine flag variety attached to level $J$ in the usual
sense in equal characteristic; in the sense of Zhu \cite{Zhu17} and Bhatt-Scholze \cite{BS17} in mixed characteristics. Since the underlying topological space of $\CF l^\bG_J$ is Jacobson, any locally closed subscheme is uniquely determined by its $\kappa$-valued points, hence we freely identify $X^\bG(\mu,b)_J$ with the above set of its geometric points. These varieties play an important role in the study of the special fiber of both Shimura varieties and moduli spaces of shtukas, see \cite{RV14}, \cite{Vi18}.

By He's work on the Kottwitz-Rapoport conjecture in \cite{He16a}, $X^\bG(\mu,b)_J\neq \emptyset$ if and only if $[b]$ is neutrally acceptable, i.e. $[b] \in B(\bG,\mu)$, which we assume from now on. The next fundamental question concerns the dimension of these geometric objects. The first results in this direction was obtained when $\bG$ is an unramified group and $\breve{\CJ}$ is a hyperspecial subgroup. In this case, the conjectural dimension formula put forth by Rapoport in \cite{Rap05} was established in \cite{GHKR06} and \cite{Vi06} for split groups, and in \cite{Ham15} and \cite{Zhu17} for general quasi-split unramified groups. On the other extreme lies the case of Iwahori level structure. Here we only have a partial understanding. For a quasi-split group $\bG$ and its Iwahori subgroup $\breve{\CI}$, a dimension formula was first obtained in \cite{HY21} under certain genericity hypotheses, namely that $\mu$ is \textit{superregular} and $b$ is \textit{sufficiently small} compared to $\mu$; the first of these conditions was subsequently weakened to $\mu$ being only regular in \cite{Sad23}. Note that by varying the level structure, we obtain transition maps $\pi_{J,J'}: \CF l_{J} \ra \CF l_{J'}$ whenever $J \subset J'$, thus inducing $\pi_{J,J'}: X^\bG(\mu,b)_J \ra X^\bG(\mu,b)_{J'}$; by \cite{He16a}, the latter map is surjective, hence $\dim X^\bG(\mu,b)_J \geq \dim X^\bG(\mu,b)_{J'}$. 

The purpose of this note is to prove a dimension formula in the case of an arbitrary parahoric level for split groups that interpolates these two known formulas.
\begin{theorem}\label{main-thm}
Suppose that $\bG$ is a split group and let $[b] \in B(\bG,\mu)$. Then
    \begin{equation}\label{dim-formula}
        \dim X^\bG(\mu,b)_J 
     =\<\rho, \underline \mu-\nu(b)\>-\frac{1}{2}\text{def}(b)+\frac{1}{2}[\{\ell(w_0)-\ell_R(w_0)\}-\{\ell(w_{J})-\ell_R(w_J)\}],
    \end{equation}

    in the following cases:
    \begin{enumerate}[(i)]
        \item $\mu$ is $2$-regular, $\mu \geq \nu(b)+\text{wt}(w_0,1)$ and $J=\emptyset$ is the Iwahori level structure.
        \item  $\mu$ is $4$-regular, $\mu \geq \nu(b)+2\rho^\vee+\text{wt}(w_0,1)$ and $J$ avoids at least one special vertex.
        \item $\mu$ is $2\ell(w_0)+2$-regular, $\mu \geq \nu(b)+2\rho^\vee+\text{wt}(w_0,1)$ and $J$ is arbitrary.
    \end{enumerate}
    
\end{theorem}
Here $\ell_R(\cdot)$ is the reflection length of elements in finite Weyl group, and $\text{wt}(\cdot,\cdot)$ is the weight function defined in terms of its associated quantum Bruhat graph, see \cref{sec:pre} for further details. Note that in part(i) we only improve \cite{Sad23} by weakening the ``gap" condition $\mu-\nu(b)$ - which is assumed to be at least $2\rho^\vee$ in loc. sit., albeit at the cost of introducing some regularity to $\mu$; note that $\wt(w_0,1)$ is substantially smaller than $2\rho^\vee$, cf. \cite[section 5]{Sad23}, e.g. in type $A_{2n}$, we have $\wt(w_0,1)=\varpi_n^\vee+\varpi_{n+1}^\vee$.
\subsection{Strategy} 
We first discuss the proof of \cref{main-thm}\textit{[(i)]}. The assumption $\nu-\nu(b)\geq 2\rho^\vee$ essentially enters the proof of \cite{HY21}[Theorem 6.1] in asserting that the KR stratum associated to certain element in $\wtd{W}$ is top-dimensional in $X^\bG(\mu,b)_{\emptyset}$, and the dimension of this KR stratum is known only under such gap hypothesis from \cite{He21pi}. To circumvent this condition, we simply construct a different top dimensional element (one in the antidominant chamber) even under the weaker gap hypothesis, by appealing to the theory of cordial elements in affine Weyl group \`a la \cite{MV21}.

Our basic strategy in proving \cref{main-thm}\textit{[(ii), (iii)]} is to adapt the methods of \cite{HY21} to arbitrary parahoric level. To that end, we note that arguments involving the EKOR strata in $X^\bG(\mu,b)_J$ - similar to those considered in loc. sit. for the KR strata in $X^\bG(\mu,b)_{\emptyset}$ - reduces our problem in \cref{sec:recollect} to one of minimizing distances in the quantum Bruhat graph; more specifically, under the conditions in the statement of the theorem, determining the dimension is equivalent to finding $$\min\{d(x,xw_0): x\in ~^JW\},$$ where $^JW$ is the semi-affine quotient of the finite Weyl group $W$ by $J$ as introduced in \cite{Sch24}; in case $\breve\CJ$ is contained in a fixed hyperspecial subgroup, this is the set of minimal length elements in $W_J\backslash W$.

However, explicit construction of an minimizing element is more complicated when $J\neq \emptyset$. We recall that in the case of Iwahori level structure, it was established in \cite{HY21} that this minimum is realized in the smaller subset of elements satisfying $x\leq xw_0$; since distance in the quantum Bruhat graph between two such elements is simply the difference of length, the previous optimization problem simplifies to finding $\max\{\ell(x): x\leq xw_0\}$ - which is more amenable to an inductive argument. We carry out a similar argument in \cref{sec-abc} Cartan types $A_n, B_n, C_n$ when $J\subset \BS$; however this approach does not seem to work nicely in type $D_n$, and in fact the previous observation about the minimum occuring at an element $x$ with $x<xw_0$ is false in general, when $J$ contains an affine simple reflection.

Hence, we resort to a different method in \cref{sec:alter} to handle the optimization problem in an uniform manner: by relating the long elements of the finite Weyl groups $W_J$ and $W$. We show that for any subset $J$ of affine simple roots, one can find a suitable subset $I$ of finite simple roots such that $w_0$ decomposes into product of certain $(~^JW)^{-1}$-conjugate of (linearization of) the long element of $w_J$ and the long element $w_I$ in a $\ell_R$-additive way. This is a proposition of independent interest, and we deduce it by combining explicit decompositions of $w_0$ listed in \cite{Sad23}, with classification of certain involutions in terms of long elements of parabolic subgroups carried out in \cite{zib23}.

\begin{acknowledgement}
The author thanks Xuhua He for his encouragement to work on this project and Mita Banik for her help with SageMath computations.
\end{acknowledgement}

\section{Preliminaries}\label{sec:pre}
\subsection{Group-theoretic notations}\label{sec:notation} 

Recall that $\bG$ is a connected reductive quasi-split group over a non-archimedean local field $F$. Let $\breve F$ be the completion of the maximal unramified extension of $F$ and $\s$ be the Frobenius morphism of $\breve F/F$. The residue field of $F$ is a finite field $\mathbb F_q$ and the residue field of $\breve F$ is the algebraically closed field $\bar{\mathbb F}_q$. We use the same symbol $\s$ for the induced Frobenius morphism on $\breve G:=\bG(\breve F)$. Let $S$ be a maximal $\breve F$-split torus of $\bG$ defined over $F$, which contains a maximal $F$-split torus. Let $\mathcal A$ be the apartment of $\bG_{\breve F}$ corresponding to $S_{\breve F}$. We fix a $\s$-stable alcove $\mathfrak a$ in $\mathcal A$, and let $\breve \CI \subset \breve G$ be the Iwahori subgroup corresponding to $\mathfrak a$. Then $\breve \CI$ is $\s$-stable.

Let $T$ be the centralizer of $S$ in $\bG$. Then $T$ is a maximal torus. We denote by $N$ the normalizer of $T$ in $\bG$. The \emph{Iwahori--Weyl group} (associated to $S$) is defined as $$\wtd{W}= N(\breve F)/T(\breve F) \cap \breve \CI.$$ 

We denote by $\ell$ the length function on $\tW$ determined by the base alcove $\mathfrak a$ and denote by $\tilde \BS$ the set of simple reflections in $\tW$. Let $W_{\text{aff}}$ be the subgroup of $\wtd{W}$ generated by $\tilde \BS$. Then $W_{\text{aff}}$ is an affine Weyl group. Let $\Omega \subset \wtd{W}$ be the subgroup of length-zero elements in $\wtd{W}$. Then $\wtd{W}=W_{\text{aff}} \rtimes \Omega.$ Note that this action of $\Omega$ preserves $\wtd{\BS}$. For any $J \subset \wtd{\BS}$, let $\wtd{W_J} \subset \wtd{W}$ be the subgroup generated by the simple reflections in $J$ and by $^J\wtd{W}$ the set of minimal length elements for the cosets $\wtd{W}_J\backslash \wtd{W}$.

Let $W=N(\breve F)/T(\breve F)$ be the relative Weyl group. We denote by $\BS$ the subset of $\wtd{\BS}$ consisting of simple reflections generating $W$. We let $\Phi$ (resp. $\D$, $\D_{\text{aff}}$) denote the set of roots (resp. simple roots, simple affine roots). We write $\Gamma$ for $\Gal(\bar F/F)$, and write $\Gamma_0$ for the inertia subgroup of $\Gamma$. Then fixing a special vertex of the base alcove $\mathfrak a$, we have the splitting $$\wtd{W}=X_*(T) \rtimes W=\{t^{ \l} w;  \l \in X_*(T), w \in W\}.$$ 

For an element $x\in \wtd{W}$ we denote by $\overline{x}$ its projection to $W$; similarly, for a root $\a\in \D_{\text{aff}}$ we denote its spherical part by $\overline{\alpha}$. For an irreducible Weyl group $W$ of rank $n$, we follow the labeling of roots as in \cite{Bou} and we usually write $s_i$ instead of $s_{\a_i}$, where $\D=\{\a_i: 1 \leq i \leq n\}$. The additional element in $\D_{\text{aff}}$ is then the affine simple root $-\theta+1$, where $\theta$ is the longest root in $\Phi$, and we denote the corresponding reflection in $\wtd{\BS}$ by $s_0$. We will also sometimes denote by $\a_s$ the affine simple root corresponding to $s\in \wtd{\BS}$. For $J\subset \wtd{\BS}$ we let $w_J$ be the longest element of the corresponding parabolic subgroup and write $w_0=w_{\BS}$. Let $\rho$ be the dominant weight with $\<\a^\vee, \rho\>=1$ for any $\a \in \D$. Let $\{\varpi_i^\vee: 1\leq i\leq n\}$ be the set of fundamental coweights. If $\varpi^\vee_i$ is minuscule,
we denote the image of $t^{ \varpi_i^\vee}$ under the projection $\wtd{W} \ra \Omega$ by $\t_i$; then conjugation by $\t_i$ is a length preserving automorphism of $\wtd{W}$ which we denote by $\text{Ad}(\t_i)$.

For any dominant coweight $\lambda$ we define $\text{depth}(\lambda)=\min\{\<\a,\lambda\>: \a\in \D\}$. Given $X\subset \wtd{W}, k\in \BZ_{\geq 0}$, we denote by $X_{>k}$ (resp. $X_{\leq k}$) the subset of $X$ whose elements have associated dominant translation of depth at least $k+1$ (resp. at most $k$); we call such elements to be $k$-regular. 

\subsection{The $\s$-conjugacy classes of $\breve G$} We say that two elements $b, b' \in \Breve{G}$ are $\s$-conjugate if there is some $g\in \Breve{G}$ such that $b'=g b \s(g) \i$. Let $B(\bG)$ be the set of $\s$-conjugacy classes on $\breve G$. By the work of Kottwitz in \cite{Ko85} and \cite{Ko97}, any $\s$-conjugacy class $[b]$ is determined by two invariants: 
\begin{itemize}
	\item The Kottwitz point $\k([b]) \in \pi_1(\bG)_{\G}$, the set of $\G$-coinvariants of the Borovoi fundamental group $\pi_1(\bG)$; 
	
	\item The Newton point $\nu([b]) \in ((X_*(T)_{\Gamma_0, \BQ})^+)^{\langle\sigma\rangle}$, the set of $\<\s\>$-invariants of the intersection of $X_*(T)_{\Gamma_0}\otimes \BQ=X_*(T)^{\Gamma_0}\otimes \BQ$ with the set $X_*(T)_\BQ^+$ of dominant elements in $X_*(T)_\BQ$.
\end{itemize}

We denote by $\le$ the dominance order on $X_*(T)_\BQ^+$, i.e., for $\nu, \nu' \in X_*(T)_\BQ^+$, we have $\nu \le \nu'$ if and only if $\nu'-\nu$ is a non-negative (rational) linear combination of positive coroots over $\breve F$. The dominance order on $X_*(T)_\BQ^+$ extends to a partial order on $B(\bG)$. Namely, for $[b], [b'] \in B(\bG)$, we say that $[b] \le [b']$ if and only if $\k([b])=\k([b'])$ and $\nu([b]) \le \nu([b'])$.

\subsubsection{Neutrally acceptable elements}\label{sec:bgmu}
Let $\{\mu\}$ be a conjugacy class of cocharacters over $\bar{F}$. Choose $\mu$ be a dominant representative of $\{\mu\}$ and denote by $\underline{\mu}$ its image in $X_*(T)_{\G_0}$. We define the \textit{$\mu$-admissible set}
\begin{equation*}
\Adm(\mu)=\{w \in \wtd{W};\ w \le t^{x(\underline \mu)} \text{ for some }x \in W\}.
\end{equation*}
Let $\breve \CJ$ be a standard $F$-rational parahoric subgroup of $\bG(\breve F)$, i.e., a $\s$-invariant parahoric subgroup that contains $\Breve{I}$. This corresponds to a spherical type $J$, i.e. a subset $J\subset \wtd{\mathbb{S}}$ such that $\wtd{W}_J$ is finite. Set $^J\text{Adm}(\mu)=\Adm(\mu)\cap ~^J\wtd{W}$. 
We also have the set of \textit{neutrally acceptable elements}
\begin{equation*}
B(\bG, \mu)=\{ [b]\in B(\bG)\mid \kappa([b])=\mu^\natural, \nu([b])\leq \mu \} .
 \end{equation*}
Here $\mu^\natural$ denotes the common image of $\mu\in\{\mu\}$ in $\pi_1(\bG)_{\s}$.

The set $B(\bG, \mu)$ naturally inherits the partial order from $B(\bG)$.

\subsection{Affine Deligne-Lusztig varieties}
The \emph{affine Deligne-Lusztig variety} associated to  $w \in \wtd{W}$ and $b \in \breve G$ is a locally closed subscheme of the partial affine flag variety of $\bf G$, locally of finite type over $\overline{\mathbb F}_p$ and of finite dimension, with the set of geometric points given by 
\begin{equation*}
    X_{w}(b)=\{g\in \breve G/\breve \CI: g^{-1}b\s(g) \in \breve\CI w \breve\CI\}.
\end{equation*}

If $F$ is of equal characteristic, then by affine flag variety we mean the ``usual'' affine flag variety; in the case of mixed characteristic, this notion should be understood in the sense of perfect schemes, as developed by Zhu~\cite{Zhu17} and by Bhatt and Scholze~\cite{BS17}.

Let $J$ be a spherical type. We set
\begin{equation*}
X^\bG(\mu,b)_J=\{g \breve \CJ \in \breve G/\breve \CJ; g \i b \s(g) \in \breve \CJ \Adm(\mu) \breve \CJ\}.
\end{equation*}
We usually omit the adoration $\bG$ from the notation and also write $X(\mu,b)$ for $X(\mu,b)_\emptyset$, which is then an union of affine Deligne-Lusztig varieties. Settling the Kottwitz-Rapoport conjecture made in \cite{KR03} and \cite{Rap05} about the non-emptiness pattern for $X(\mu,b)_J$, He proves the following result in \cite{He16}.
\begin{theorem}\label{KR-conjecture}\cite[theorem A]{He16a}
$X(\mu,b)_J \neq \emptyset$ if and only if $[b] \in B(\bG,\mu)$.
\end{theorem}

\subsection{Quantum Bruhat graphs}
We recall the quantum Bruhat graph introduced by Brenti, Fomin, and Postnikov in the context of quantum cohomology ring of complex flag variety, see \cite{BFP99}, also \cite{FGP97}. By definition, a \emph{quantum Bruhat graph} $\G_{\Phi}$ is a directed graph with 
\begin{itemize}
\item vertices given by the elements of $W$; 

\item upward edges $w\rightharpoonup w s_\a$ for some $\a \in \Phi^+$ with $\ell(w s_\a)=\ell(w)+1$;

\item downward edges $w \rightharpoondown w s_\a$ for some $\a \in \Phi^+$ with $\ell(w s_\a)=\ell(w)-\< 2 \rho, \a^\vee\>+1$. 
\end{itemize}

\smallskip

The \emph{weight of a path} in $\G_{\Phi}$ is defined to be the sum of weights of the edges in the path. For any $x,y\in W$, we denote by $d(x, y)$ the minimal length among all paths in $\G_{\Phi}$ from $x$ to $y$. Any path between $x$ and $y$ affording $d(x,y)$ as its length is called a \textit{shortest path} between them. Then the following lemma defines a function $\text{wt}: W \times W \ra \BZ_{\geq 0} \Phi^\vee$.

\begin{lemma}\label{wt-x-y} \cite[Lemma 1]{Pos05}, \cite[Lemma 6.7]{BFP99}
Let $x,y\in W$. 
\begin{enumerate}
\item There exists a directed path (consisting of possibly both upward and downward edges) in $\G_{\Phi}$ from $x$ to $y$.

\item Any two shortest paths in $\G_{\Phi}$ from $x$ to $y$ have the same weight, which we denote by $\text{wt}(x, y)$. 
\item Any path in $\G_{\Phi}$ from $x$ to $y$ has weight $\ge \wt(x,y)$.
\end{enumerate}
\end{lemma}

\begin{lemma}\label{wt-d}\cite[\S 4.2]{HY21}, \cite[\S 2.5]{HN21}
Let $x,y \in W$. Then
\begin{enumerate}
    \item We have $\<\wt(x,y),\a\> \leq 2$, for any simple root $\a$.
    \item $\ell(y)-\ell(x)=d(x,y)-\<2\rho,\wt(x,y)\>$.
\end{enumerate}

\end{lemma}
\begin{lemma}\label{d-subtract}
    $d(x,xy)=\ell_R(y)$ if $\ell(xy)=\ell(x)-\ell(y)$.
\end{lemma}
\begin{proof}
    We have $\text{wt}(x,xy)=\text{wt}(x^{-1}\lhd xy,1)$ by \cite{Sad23}[Corollary 3.3], where $\lhd$ denotes the left downward Demazure product, see \cite{HL15}[\S 2.2]. Thus, $d(x,xy)=\ell(xy)-\ell(x)+\ell(x^{-1}\lhd xy)+\ell_R(x^{-1}\lhd xy)$ by \cite{Sad23}[Lemma 8.4] and \cref{wt-d}[(2)]. Now, $x^{-1}\lhd xy=y$ if and only if $\ell(y)=\ell(x^{-1})-\ell(xy)$, but in that case the first equation simplifies to $d(x,xy)=\ell_R(y)$.
\end{proof}

\section{Recollections from previous works}\label{sec:recollect}
\subsection{Some standard reductions}
Note that if $\bG=\bG_1\times \bG_2, b=(b_1,b_2), \mu=(\mu_1,\mu_2), \breve \CJ=\breve \CJ_1 \times \breve \CJ_2$, then by definition $X^\bG(\mu,b)_{J} \simeq X^{\bG_1}(\mu_1,b_1)_{J_1} \times X^{\bG_2}(\mu_2,b)_{J_2}$; similarly, the terms on the RHS of the dimension formula \cref{dim-formula} are additive as well. Hence, we may assume from now on that $\bG$ is quasi-simple and split over $F$.

Next, we argue that in certain cases we may assume that the parahoric subgroup is contained in a hyperspecial subgroup, i.e. $J\subset \mathbb{S}$. To do so, we need the following observation from \cite{He14}[Lemma 4.3].
\begin{lemma}\label{same-dim}
    Let $w\in \wtd{W}, \tau \in \Omega$. Then $X_w(b)$ is isomorphic to $X_{^\tau w }(b)$.
\end{lemma}

Note that $s_i w\geq w \iff ~^\tau s_i  ~^\tau w\geq ~^\tau w$, and thus $\text{Ad}(\tau)(^J \wtd{W})=^{\text{Ad}(\tau)(J)} \wtd{W}$. Also, $^\tau \text{Adm}(\mu)=\text{Adm}(\mu)$, whence $$\text{Ad}(\tau)(^J \wtd{W} \cap \text{Adm}(\mu))=^{\text{Ad}(\tau)(J)} \wtd{W} \cap \text{Adm}(\mu)=^{\text{Ad}(\tau)(J)}\Adm(\mu).$$

It is proved in \cite{GH15}[\S 3.4] (see also \cite{HR17}[Theorem 6.21]) that $$\dim X(\mu,b)_J=\max\limits_{w\in ^J\text{Adm}(\mu)} \dim X_w(b).$$ 

This comes from a decomposition of $X(\mu,b)_J$ that is analogous, in terms of Shimura varieties, to the decomposition of a
Newton stratum into its intersections with the EKOR strata, cf. \cite{HR17}.

In view of \cref{same-dim}, we then have $\dim X(\mu,b)_J=\dim X(\mu,b)_{\text{Ad}(\tau)(J)}$. Hence, if there exists a simple reflection $s$ corresponding to a special vertex of $\mathfrak a$ so that $s\notin J$, we may use a suitable element $\t \in \Omega$ so that $\text{Ad}(\tau)(J) \subset \BS$ and then focus on the dimension problem for this new level structure. For example, this assumption on $J$ is satisfied for any parahoric level structure in type affine Weyl group of type $A_n$, so here we may assume $J\subset \BS$.
\subsection{Virtual dimension formula}
Recall that $w \in \wtd W$ can be
written in a unique way as $w = ut^\l v$ with $\l$ dominant, $u, v \in W$ such that $t^\l v \in ~^\BS\wtd W$. We then set $\eta(w) = vu, p_l(w)=u$. 

Let $\mathbf J_b$ be the reductive group over $F$ with $\mathbf J_b(F)=\{g \in \breve G; g b \s(g) \i=b\}.$ Then the \emph{defect} of $b$ is defined by $\rm{def}_{\bG}(b)=\rank_F \bG-\rank_F \mathbf J_b$, where $\rank_F$ denotes the $F$-rank.

Following \cite[Section 10.1]{He14}, we define the \emph{virtual dimension} for $X_w(b)$ to be \[d_{w}(b)=\frac 12 \big( \ell(w) + \ell(\eta(w)) -\rm{def}_{\bG}(b)  -\<2\rho, \nu([b])\>).\] 

It is proved in \cite[Corollary 10.4]{He14} that $\dim X_w(b) \leq d_w(b)$ whenever $\k(w)=\k([b])$. 

For $J\subset \wtd{\BS}$, define $W_J$ to be the subgroup of $W$ generated by $\{\overline{s}: s\in J\}$ and $^JW:=\{w\in W: w^{-1}\overline{\a_s} \in \Phi^+, ~\forall s\in J\}$; this is notationally consistent, i.e. when $J\subset \BS$, the latter is the set of minimal length elements in $W_J\backslash W$. Given any $J\subset \wtd{\BS}, w\in W$, it is proved in \cite{Sch24}[Lemma 3.7] that we can find $w_1\in W_J, w_2\in ~^JW$ so that $w=w_1w_2$. 
\begin{lemma}\label{left-part}
    Let $\mu$ be dominant regular; then $p_l(^J \text{Adm}(\mu))=~^JW$ if $J \subset \mathbb{S}$, and $p_l(^J \text{Adm}(\mu)_{>1})=~^JW$ otherwise.
\end{lemma}
\begin{proof}
   Let $w\in ~^J \text{Adm}(\mu)$ and write $w=xt^\lambda y$ with $t^\lambda y \in ~ ^\mathbb{S}\wtd{W}$. For $s_i\in \mathbb{S}$,  $s_iw>w$ then implies $s_ix>x$, thus $x^{-1}\a_i\in\Phi^+$. Now assume $\text{depth}(\mu)>1$. We have $s_0w>w$ if and only if $w^{-1}(-\theta,1)=(-(xy)^{-1}\theta, 1-\<x\lambda ,\theta\>)$ is a positive affine root, i.e., $\<x\lambda, \theta\>=\<\lambda,x^{-1}\theta\>\leq 1$ or $0$, depending on $-(xy)^{-1}(\theta) \gtrless 0$. Then we must have $x^{-1}(-\theta) \in \Phi^{+}$, since otherwise $\text{depth}(\lambda)\leq 1$.

     For the converse direction we note $w':=xt^{\mu}w_0\leq t^{w_0\mu}$ for any $x\in W$, by regularity of $\mu$. If $s_ix>x$ for $s_i\in \mathbb{S}$ then $s_iw'> w'$; also, the argument in the last paragraph shows $s_0w'>w'$ if $x^{-1}(-\theta)\in \Phi^+$, by the depth hypothesis on $\mu$. Hence, $w'\in ~^J \text{Adm}(\mu)$.
\end{proof}

Following \cite{HY21}, we define $d_{^J\text{Adm}(\mu)}(b):=\max\limits_{w\in ^J\text{Adm}(\mu)} d_w(b)$. Therefore, $\dim X(\mu,b)_J \leq d_{^J\text{Adm}(\mu)}(b)$.
\begin{proposition}\label{d-adm}
    We have $$d_{^J\text{Adm}(\mu)}(b)= \<\rho,\mu-\nu([b])\>-\frac{1}{2}\mathrm{def}_{\bG}(b) +\frac{1}{2}\ell(w_0)-\frac{1}{2}\min \{d_\G(x,xw_0): x\in ~^JW\},$$ if $\text{depth}(\mu)\geq 4$ in case $J\subset \BS$, or if $\text{depth}(\mu) \geq 2\ell(w_0)+2$ otherwise.
\end{proposition}

\begin{proof}
The proof scheme here is identical to that in the case of Iwahori level structure in \cite{HY21}[\S 4.2-4.3]. We only need to explain the depth hypothesis requirement and justify the index set of the last summand. Note that the proof in \cite{HY21}[Proposition 4.4] was carried out under the assumption $\text{depth}(\mu)>>0$, which was later relaxed in \cite{Sad23}[Proposition 7.1] to $\mu$ being just regular; this latter improvement in part relied on utilizing the explicit construction of certain elements in $\text{Adm}(\mu)\cap Wt^{\mu}W$ coming from \cite{HY21}[\S 5.6]. In the absence of this latter sort of explicit construction in general\footnote{However, this is available when $J\subset \BS$ and $W$ is irreducible of type $A_n, B_n/C_n$, by \cref{sec-abc}.}, we may consider the element $\wtd{w}:=x't^{\mu-\text{wt}(x',x'w_0)}w_0x'^{-1}$ for some $x'\in ~^JW$; note that $\mu-\text{wt}(x',x'w_0)$ is $2$-regular by \cref{wt-d}[(1)] due to our imposed depth hypothesis. Then $\wtd{w}$ is in $^J\Adm(\mu)$ by \cite{Sch24}[Proposition 4.12] and \cref{left-part}, and we can use this element to replace the arguments in the last paragraph of \cite{Sad23}[Proposition 7.1].

Now let us discuss the index set. If $J\subset \BS$, \cref{left-part} does the job. Assume now that $s_0\in J$ and $\<\a_i,\mu\>\geq 2\ell(w_0)+2$ for all $\a_i\in \Delta$. Pick $x'\in ~^JW$ and consider $w':=x't^{\mu-\text{wt}(x',x'w_0)}w_0x'^{-1} \in ~^J\Adm(\mu)$. We argue
    
(a) if $w \in ~^J \text{Adm}(\mu)_{\leq 1}$, then $\ell(w)<\ell(w')$.

Write $w=xt^{\lambda}y$ as before, then $\lambda\leq \mu$; assume that $\<\alpha_j,\lambda\>\leq 1$ for some $\alpha_j \in \Delta$. Note that $\ell(w)\leq \<2\rho,\lambda\>+\ell(w_0)$ and $\ell(w')=\ell(x')+\<2\rho,\mu-\text{wt}(x',x'w_0)\> -\ell(w_0x'^{-1}) = \<2\rho,\mu\> -d(x',x'w_0)$ by \cref{wt-d}[(2)], thus $\ell(w')\geq \<2\rho,\mu\>-\ell(w_0)$. Then $$\ell(w')-\ell(w) \geq \<2\rho,\mu-\lambda\>-2\ell(w_0)\geq \<2\varpi_i,\mu-\lambda\>-2\ell(w_0)\geq \<\a_i,\mu-\lambda\>-2\ell(w_0)>0.$$
 
Hence, (a) is proved. This gives $d_w(b)<d_{w'}(b)$ - thereby implying $d_{^J\text{Adm}(\mu)}(b)=\max\limits_{w\in ^J\text{Adm}(\mu)_{>1}} d_w(b)$, and finishing off the argument by appealing to \cref{left-part}.

\end{proof}

\subsection{Minimizing distance on the quantum Bruhat graph}\label{easy-ineq}
We now state the key technical result on quantum Bruhat graph that will be used in \cref{sec-proof}. 
\begin{theorem}\label{min}
    Let $J\subset \wtd{\BS}$ be a spherical type. Then $\min\limits_{x\in ^J W}d(x,xw_0)=\ell_R(w_0)+\ell(w_J)-\ell_R(w_J)$.
\end{theorem}
Here $\ell_R(w_0)$ (resp. $\ell_R(w_J)$) is the reflection length of the long element $w_0$ in $W$ (resp. $w_J$ in the parabolic subgroup $\wtd{W}_J$ of $\wtd{W}$), i.e. the smallest number $l$ such that $w_0\in W$ (resp. $w_J\in \wtd{W}_J$) can be written as a product of $l$ reflections in $W$ (resp. in $\wtd{W}_J$), see \cite{HY21}[\S 5.1] for details. We note that $\ell_R(\overline{w_J})=\ell_R(w_J)$.
\begin{remark}
    When $J=\emptyset$, i.e. in the case of the Iwahori level, the conclusion of \cref{min} is known from \cite{HY21}[Theorem 5.1]. However, the proof in loc. sit. could be considerably simplified. Indeed, $d(x,xw_0)\geq \ell_R(w_0)$ for any $x$, and if $x=w_0$ then $d(x,xw_0)=d(w_0,1)=\ell_R(w_0)$ by \cite{Sad23}[\S 5], thereby concluding $\min\limits_{x\in W} d(x,xw_0)=\ell_R(w_0)$.
\end{remark}

In this subsection we establish the lower bound. This relies on the following lemma.

\begin{lemma}\label{key-lemma}
    $d(x,xw_0)=\ell(y)+d(x, \overline{y}xw_0)$ whenever $x\in ~^JW$ and $y\in \wtd{W}_J$
\end{lemma}
\begin{proof}
We will induct on $\ell(y)$. By \cref{left-part}, we may write $w=xt^{\lambda}z \in ~^J\wtd{W}$ for some $z\in W$ and cocharacter $\lambda$ of depth $>>0$.

Clearly, the claim is trivial for $y=1$. Assume now the claim is true for $y'\in \wtd{W}_J$ and let $y=sy'>y'$, with $s\in J$; then $sy'w>y'w$. Note that $y'w=\overline{y'}xt^{\lambda'}z$, where $\lambda'$ still has depth $>>0$. Arguing as in the proof of \cref{left-part}, we thus get $(\overline{y'}x)^{-1}(\alpha_i)\in \Phi^+$ if $s=s_{\alpha_i}$ is a finite simple reflection, or $(\overline{y'}x)^{-1}(-\theta)\in \Phi^+$ if $s=s_{(-\theta,1)}$ is the affine simple reflection; note that $x^{-1}\alpha_i \in \Phi^+$ and $x^{-1}(-\theta) \in \Phi^+$ accordingly in these cases. Thus by \cite{len15}[Lemma 7.7], $d(x, \overline{y'}xw_0)=1+d(x, \overline{sy'}xw_0)$ in either case, whence $$d(x,xw_0)=\ell(y')+d(x, \overline{y'}xw_0)=\ell(y')+1+d(x, \overline{sy'}xw_0)=\ell(y)+d(x, \overline{y}xw_0).$$

\end{proof}

In particular, \cref{key-lemma} implies $d(x,xw_0)=\ell(w_J)+d(x, \overline{w_J}xw_0)$. Suppose now for contradiction that there exists $x\in~ ^J W$ with $$d(x,xw_0) < \ell_R(w_0)+\ell(w_J)-\ell_R(w_J).$$ 
We then get 
\begin{align*}
    \ell_R(w_0)-\ell_R(w_J) > d(x,\overline{w_J}xw_0) \geq \ell_R(x^{-1}\overline{w_J}xw_0)\\ \geq \ell_R(w_0)-\ell_R(x^{-1}\overline{w_J}x) = \ell_R(w_0)-\ell_R(w_J),
\end{align*}
a contradiction. Hence,
\begin{equation}\label{lower-bd}
    \min\limits_{x\in ^J W}d(x,xw_0)\geq \ell_R(w_0)+\ell(w_J)-\ell_R(w_J).
\end{equation}

\section{Explicit constructions for type $A_n, B_n /C_n$}\label{sec-abc}

In this section, for an irreducible Weyl group $W$ of type $A_n, B_n/C_n$ and $J\subset \BS$, we set out to prove

(a) there exists $x \in ~^JW$ with $x\leq xw_0$ and $\ell(x)=\frac{1}{2}\{(\ell(w_0)-\ell_R(w_0))-(\ell(w_J)-\ell_R(w_J))\}$.

This is analogous to the construction discussed in \cite{HY21}[section 5.6] for the case of $J=\emptyset$. Our argument is based on an downward induction on $|J|$. If $J=\mathbb{S}$, then $x=1$ satisfies the condition in (a). Let $J'\subset \BS$ with $J'=J\setminus \{s_j\}$ for some $1\leq j \leq n$. For induction, note that $$^{J'} W=\coprod\limits_{i \in ^{J'}W_{J}} i \cdot ~^JW.$$ 

By inductive hypothesis, there exists $x\in ~^JW$ as asserted in (a). Now, suppose that there exists $i\in ^{J'}W_{J}$ such that 
\begin{enumerate}[(i)]
    \item $i < iw_{J}$, and
    \item $\ell(i)=\frac{1}{2}\{(\ell(w_J)-\ell_R(w_J))-(\ell(w_{J'})-\ell_R(w_{J'})\}$.
\end{enumerate}

Upon taking minimal length left coset representatives, the inequality $x\leq xw_0$ gives $x \leq ~^J(xw_0)=w_Jxw_0$; since $x, w_Jxw_0 \in ^JW$ and $i,iw_J \in W_J$, we have $\ell(ix)=\ell(i)+\ell(x), \ell(iw_J\cdot w_Jxw_0)=\ell(iw_J)+\ell(w_Jxw_0)$. We thus obtain $ix\leq iw_J\cdot w_Jxw_0=ixw_0$. Furthermore, $\ell(ix)=\ell(i)+\ell(x)=\frac{1}{2}\{(\ell(w_0)-\ell_R(w_0))-(\ell(w_{J'})-\ell_R(w_{J'}))\}$. Therefore, $ix \in ~^{J'}W$ is the desired element for which the lower bound for the function $y \mapsto d(y,yw_0)$ on $^{J'}W$ is achieved, and this completes the induction step.

Now we do a case-by-case analysis to give explicit description of such elements $i \in ~^{J'}W_{J}$. Note that it suffices for this purpose to reduce to the case where $J$ is irreducible. Since we only consider Weyl groups of type $A_n, B_n/C_n$ in this subsection, the relevant subdiagram $J$ is also of the same type. To simplify notation, we write $s_{[a,b]}=s_{a}s_{a+1}\cdots s_b$ for $1\leq a<b \leq n$.

\subsubsection{$J$ is of type $A_n$} We need to consider only the case of $1\leq j \leq \frac{\lceil n \rceil}{2}$, because conjugation by $w_0$ is a symmetry of the Dynkin diagram taking $j$ to $n+1-j$. We have $\ell(w_J)-\ell_R(w_J)=\frac{n(n+1)}{2}-\lceil \frac{n}{2}\rceil$.

If $j=1$, then $J'\simeq A_{n-1}$. In this case, it is already determined in \cite{HY21} that $i=s_{[1,\lfloor \frac{n}{2} \rfloor]}$ fits the bill.

If $2\leq j\leq \frac{\lceil n \rceil}{2}$, then $J'\simeq A_{j-1}\times A_{n-j}$, and thus $$\ell(w_{J'})-\ell_R(w_{J'})=\frac{j(j-1)}{2}+\frac{(n-j)(n-j+1)}{2}-(\lceil\frac{j-1}{2}\rceil+\lceil \frac{n-j}{2} \rceil).$$ Then $i=s_{[j,n]}s_{[j-1,n-1]}\cdots s_{[j+1-\lfloor \frac{j}{2}\rfloor, n+1-\lfloor\frac{j}{2}\rfloor]}\delta$, where 
\begin{equation*}
    \delta=
        \begin{cases}
          &1, \text{ if }j \text{ is even };\\
          & s_{[\lceil\frac{j}{2}\rceil, \lfloor\frac{n}{2}\rfloor]},  \text{ if }j \text{ is odd}.
        \end{cases}
    \end{equation*}
It is an easy calculation to check that $\ell(i)$ is equal to $\frac{1}{2}(nj-j^2+j)$ if $j$ is even, and $\frac{1}{2}(nj-j^2+j-1)$ if $j$ is odd.

\subsubsection{$J$ is of type $B_n$ or $C_n$} We have $\ell(w_J)-\ell_R(w_J)=n^2-n$. If $j=1$, then $J'\simeq B_{n-1}$ (resp. $C_{n-1}$). In this case, it is already determined in \cite{HY21} that $i=s_{[1,n-1]}$ fits the bill.

If $2\leq j\leq n$, then $J'\simeq A_{j-1}\times B_{n-j}$, and thus $$\ell(w_{J'})-\ell_R(w_{J'})=\frac{j(j-1)}{2}+(n-j)^2-(\lceil\frac{j-1}{2}\rceil+n-j).$$ 

Then $i=\prod\limits_{\kappa=0}^{\lceil \frac{j}{2}\rceil-1}s_{[j-\kappa,n]} \prod\limits_{\eta=0}^{\lceil \frac{j}{2} \rceil-1}s_{[\lceil \frac{j}{2} \rceil-\eta, n-\gamma-2\eta]}$, where 
\begin{equation*}
    \gamma=
        \begin{cases}
          &2, \text{ if }j \text{ is even };\\
          &1,  \text{ if }j \text{ is odd}.
        \end{cases}
    \end{equation*}
An easy calculation verifies that $\ell(i)=2jn-\frac{3j^2+j}{2}+\lceil \frac{j-1}{2}\rceil$.

\begin{remark}
The discussion above shows that $$\min\limits_{x\in ~^JW}d(x,xw_0)=\min\limits_{x\in ~^JW, x\leq xw_0}d(x,xw_0)=\ell(w_0)-2\max\limits_{x\in ~^JW, x\leq xw_0}\ell(x).$$
This is in similar spirit to the argument in \cite{HY21}. However, an attempt at the above sort of explicit construction results into cumbersome formula for type $D_n$. Furthermore, this is false when $s_0\in J$; for instance, take $W$ to be of type $B_4$ and $J=\{s_0,s_1,s_3,s_4\}$, then $\{x\in ~^JW: x\leq xw_0\}=\emptyset$.
\end{remark}

\section{Good decomposition of the long element in finite Weyl groups}\label{sec:alter}
In this subsection we analyze certain decomposition of the long element $w_0$ well-suited to our purpose of realizing the minimum asserted in \cref{min}.
\begin{proposition}\label{exceptional}
    Let $J\subset \wtd{\BS}$ be a spherical type. Then there exists $x \in ~^JW $ and $I \subset \BS$ such that
    \begin{enumerate}[(i)]
        \item $w_0= ~(^{x^{-1}}\overline{w}_J) w_I$;
        \item The above decomposition is $\ell_R$-additive, i.e. $\ell_R(w_0)=\ell_R(w_J)+
    \ell_R(w_I)$;
        \item $d(x,xw_I)=\ell_R(w_I)$.
    \end{enumerate}
\end{proposition}

Before proceeding with the proof of this proposition, we first establish \cref{min}.
\begin{proof}[Proof of \cref{min}]
For Weyl groups of type $A_n, B_n/C_n$, we may use the construction scheme in \cref{sec-abc} if $J\subset \mathbb{S}$ : note that $$d(x,xw_0)=\ell(xw_0)-\ell(x)=\ell_R(w_0)+\ell(w_J)-\ell_R(w_J)$$ for such $x$, thereby proving \cref{min} in these cases.

Let $J\subset \wtd{\mathbb{S}}$ be arbitrary. For the element $x\in ~^JW$ described in \cref{exceptional}, we compute $$d(x,\overline{w_J}xw_0)=d(x,x ^{x^{-1}}\overline{w_J} w_0)=d(x,xw_I)=\ell_R(w_I)=\ell_R(w_0)-\ell_R(w_J).$$ Therefore, by \cref{key-lemma} we have $d(x,xw_0)=\ell(w_J)+\ell_R(w_0)-\ell_R(w_J)$. We are done.
\end{proof}

\begin{proof}[Proof of \cref{exceptional}]
    It will be shown in \cref{type-D}-\cref{type-G2} that there exists $z \in W$ and 
$I \subset \BS$ such that $^z(w_0w_I)=\overline{w_J}$ and $\ell_R(w_J)=\ell_R(w_0w_I)=\ell_R(w_0)-\ell_R(w_I)$. Let us now establish the remaining properties in the desiderata.

\smallskip
We first show that

(a) $^x(w_0w_I)=\overline{w_J}$ for $x=^J z$. 

Choose $\z \in \wtd{W}_J$ such that $\overline{\z}=z_J \in \overline{\wtd{W}_J}=W_J$. It suffices to show that $^{\z^{-1}}w_J=w_J$ for such $\z\in \wtd{W}_J$, and hence $\overline{^{\z^{-1}}w_J}=\overline{w}_J$, i.e. $^{z_J^{-1}}\overline{w_J}=\overline{w_J}$. Suppose otherwise; then $\ell(^{\z^{-1}}w_J)< \ell(w_J)$, since $^{\z^{-1}}w_J \in \wtd{W}_J$. By $\cref{key-lemma}$, 
\begin{align*}
    d(x,xw_0)&=\ell(^{\z^{-1}}w_J)+d(x,(^{z_J^{-1}}\overline{w_J})xw_0)=\ell(^{\z^{-1}}w_J)+d(x,x (^{z^{-1}}\overline{w_J})w_0)\\&=\ell(^{\z^{-1}}w_J)+d(x,xw_I)=\ell(^{\z^{-1}}w_J)+\ell_R(w_I)\\&=\ell(^{\z^{-1}}w_J)+\ell_R(w_0)-\ell_R(w_J)<\ell(w_J)+\ell_R(w_0)-\ell_R(w_J).
\end{align*}

By \cref{lower-bd}, this is a contradiction and hence $^{z_J^{-1}}\overline{w_J}=\overline{w_J}$. 

Hence, we have shown that the set $$\Xi_{J,I}=\{x\in ~^JW: \text{ (a), (b) of the desiderata is fulfilled for the pair } (x,I)\}$$ is nonempty.

\smallskip
Next, we claim that

(b) if $x$ is an element of maximal length in $\Xi_{J,I}$, then $x\Delta_I \subset \Phi^-$. 

Suppose otherwise that $x\alpha_i > 0$ for some $\alpha_i \in \Delta_I$; then $xs_i > x$. We show that $(xs_i,I) \in \Xi_J$, contradicting the choice of $x$. It suffices to check that $^{s_i}(w_0w_I)=w_0w_I$. Note that the long element is central, unless it corresponds to a Dynkin diagram of type $A_m (m\geq 2), D_{2m+1}$ and $E_6$. This consideration alone takes care of most of the cases arising in the following calculations, and we may easily verify this condition in the remaining cases with the description of $I$ at hand.

\smallskip

Finally, we note that 

(c) $d(x,xw_I)=\ell_R(w_I)$ for the element found in (b).

For $x \in W$, define $\text{Inv}(x)=\{\alpha\in \Phi^+: x\alpha \in \Phi^-\}$; then $\text{Inv}(w_I)=\Phi_I^+$. 
The well-known length formula of product of elements in the Weyl group
$$\ell(xy)=\ell(x)-\ell(y)+2|\text{Inv}(x)^c\cap \text{Inv}(y^{-1})|$$
gives $\ell(xw_I)=\ell(x)-\ell(w_I)$ by (b), and then we deduce (c) from \cref{d-subtract}.

\end{proof}

Let us write $a\sim b$ to denote that $a,b\in W$ are in the same conjugacy class. Below we explicitly determine the assignment $J\mapsto I$ via a case-by-case analysis. For this purpose, we may reduce to the case $J\subset \BS$; indeed, since $\overline{w_J}$ is an involution in $W$, there exists $J'\subset \BS$ so that $\overline{w_J}\sim w_{J'}$, then we may apply the assignment $J'\mapsto I$ to derive the conclusion of \cref{exceptional} in this case. Note that the element $w_0w_I$ as in the proposition must be an involution and we utilize below the classification of conjugacy classes of such involutions as carried out in \cite{zib23}, often without further reference; the author in loc. sit. considers conjugacy in the centralizer subgroup $C_W(w_0)$, whereas our situation allows for more flexibility and we exploit that to ensure the additional properties claimed in the proposition. 

We explain in detail the case of type $D_n$, which is the most complicated, and give only brief sketch of arguments for types $A_n, B_n/C_n$. For the exceptional types, the tables contain the information of $(J,\ell_R(w_J))$ in the first row and $(I,\ell_R(I))$ in the second row and the roles of $(J,I)$ are symmetric. For brevity, we skip the cases with factors of type $A_k$ with $k\geq 2$; using the fact $w_{A_k}\sim \prod\limits_{\lceil \frac{k}{2}\rceil \text{ copies}}w_{A_1}$, we record the latter cases instead. Also, in type $E_n$ we use $w_{[1,6]}=w_{E_6}\sim w_{D_4}$. Finally, We have $\ell_R(w_0w_I)=\ell_R(w_0)-\ell_R(w_I)$ by a direct check of explicit $\ell_R$-reduced expression of $w_0$.

\subsubsection{Type $D_n$}\label{type-D}

Here $J \simeq A_{n_1}\times \cdots \times A_{n_r}\times D_l$ where $n_1,\cdots, n_r >0$, and either $l=0$ or $l \geq 2$. Our convention here is to consider the largest connected component of $J$ containing $\{n-1,n\}$ to be of type $D_l$, if it exists. We write $J=[i_1,i_1+n_1] \cup \cdots \cup [i_r,i_k+n_r] \cup [n-l+1, n]$, where the last interval is present only when $\{n-1,n\}\subset J$.

\smallskip 
We first show that

(a) $ w_J\sim (\prod\limits_{k \text{ copies}} w_{A_1}) \cdot w_{D_l}, \text{ with } k=\lceil \frac{n_1}{2} \rceil + \cdots + \lceil \frac{n_r}{2} \rceil.$

Now, for a Weyl group with root system of type $A_m$, its longest element $w_{A_m}$ is conjugate to $w_{0,J'}$ where $J'=[1,m]\cap \mathbb{Z}_{\text{odd}}$. Thus, for each $1\leq j \leq r$ we can find $x_j' \in W(A_{n_j})$ such that $^{x_j'} w_{0,A_{n_j}}=w_{0,J_j'}$, where $J_j'=\{i_j, i_j+2,\cdots, i_j+2(\lceil \frac{n_j}{2} \rceil -1)\}$ for $1\leq j\leq r$. Note that $\text{supp}(x_i')$ and $\text{supp}(w_{A_{n_j}})$ is orthogonal whenever $i\neq j$; thus setting $x'=x_1'\cdots x_k'$, we see that $$^{x'} w_J=(\prod\limits_{j=1}^r \prod\limits_{\kappa_j=0}^{\lceil \frac{n_j}{2} \rceil -1} s_{i_j+2\kappa_j}) \cdot w_{D_l}.$$

Next, we argue that 

(b) the first factor in the RHS above is conjugate to $s_1s_3\cdots s_{2k-1}$ by an element whose support is orthogonal to $[n-l+1,n]$. 

We work with one connected component of $J$ at a time. Suppose that $i_1>1$. Since $^{s_{\alpha_{a-1}+\alpha_{a}}}s_{a}=s_{a-1}, ^{s_{\alpha_{a-1}+\alpha_{a}}}s_b=s_b$ whenever $b\geq a+2$ or $b \leq a-3$, we can apply repeated conjugation to get $$\prod\limits_{\kappa_1=0}^{\lceil \frac{n_1}{2} \rceil -1} s_{i_1+2\kappa_1} \sim s_1 \prod\limits_{\kappa_1=1}^{\lceil \frac{n_1}{2} \rceil -1} s_{i_1+2\kappa_1} \sim s_1s_3\prod\limits_{\kappa_1=2}^{\lceil \frac{n_1}{2} \rceil -1} s_{i_1+2\kappa_1}\sim \cdots \sim s_1s_3\cdots s_{\lceil \frac{n_1}{2}\rceil}.$$

Hence, passing on to a further conjugate, we have that $$^{x''} w_J=s_1s_3\cdots s_{\lceil \frac{n_1}{2}\rceil}(\prod\limits_{j=2}^r \prod\limits_{\kappa_j=0}^{\lceil \frac{n_j}{2} \rceil -1} s_{i_j+2\kappa_j}) \cdot w_{D_l}.$$ Let $i_p$ be the minimum positive integer such that $i_p-\{i_{p-1}+2(\lceil \frac{n_{p-1}}{2} \rceil -1)\}\geq 3$. In this case, we already have $\prod\limits_{j=1}^{p-1} \prod\limits_{\kappa_j=0}^{\lceil \frac{n_j}{2} \rceil -1} s_{i_j+2\kappa_j}=s_1s_3\cdots s_{\lceil \frac{n_{p-1}}{2} \rceil}$. It is now clear that we can run a similar sequence of conjugation to move the block starting with $s_{i_p}$ too, finishing our argument.

Denote by $w_{k,l}:=w_I$ the longest element for the sub-diagram $I=\{1,3,\cdots, 2k-1, n-l+1, n-l+2,\cdots n\}$\footnote{In \cite{zib23}, this element is denoted by $c_{k,l-1}$.}. So far, we have established $w_J \sim w_{k,l}$.

Note that $2k-1\leq n-2$, unless $n$ is even and $J\sim A_{n-1}$, i.e. $J=\BS \setminus \{n-1\}$ or $J=\BS\setminus \{n\}$. Let us exclude these cases for now and consider the element $w_I$ defined as
\begin{equation*}
w_I=
\begin{cases}
&w_{k,0}, \text{ if }n=2k+l, \text{ or } n-(2k+l)=1\text{ and }l \text{ is even };\\
& w_{k,n-2k-(l-1)},  \text{ if } n-(2k+l)\geq 1 \text{ and } l \text{ is odd};\\
& w_{k,n-2k-l}, \text{ otherwise}.
\end{cases}
\end{equation*}

We have $\ell_R(w_0w_I)=\ell_R(w_0)-\ell_R(w_I)$, by the explicit decomposition of $w_0$ in \cite{Sad23}[\S 5].
\begin{itemize}
    \item In the first case, $w_I$ is a $\ell_R$-subword of $w_0$ since all simple reflections $s_i$ such that $i\leq n-2$ and $i$ is odd occurs in a $\ell_R$-reduced expression of $w_0$.
    \item In the remaining cases, the second index in the subscript of the definition of $w_I$ has same parity as that of $n$, and thus comparing the explicit expression of long element of Weyl groups of type $D_l$ (supported on $\{n-l+1, n-l+2,\cdots n\}$) and $D_n$, we see that the former is a $\ell_R$-subword of the latter.
\end{itemize}
Furthermore, $w_{I}w_0=w_0w_{I}$ and hence $w_{I}w_0$ is an involution. It follows directly from \cite{zib23} that $w_0w_I \sim w_J$ in the first two cases; for the third case, we note $w_J=w_{k,l}\sim w_{k,l+1}\sim w_0w_{k,n-2k-l}=w_0w_I$.

Now we deal with the cases where $n$ is even and $J\sim A_{n-1}$. Here
\begin{itemize}
    \item $w_0w_{\BS \setminus \{n-1\}}\sim w_{\{1,3,\cdots,n-3, n\}}, w_0w_{\BS \setminus \{n\}}\sim w_{\{1,3,\cdots,n-3,n-1\}}$, if $n=0$ mod $4$;
    \item  $w_0w_{\BS \setminus \{n-1\}}\sim w_{\{1,3,\cdots,n-3,n-1\}}, w_0w_{\BS \setminus \{n\}}\sim w_{\{1,3,\cdots,n-3, n\}}$, if $n=2$ mod $4$. 
\end{itemize}
\subsubsection{Type $A_n$} Here $J\simeq A_{n_1} \times \cdots \times A_{n_r}$. Then $w_J\sim (\prod\limits_{p \text{ copies }}w_{A_1})$ with $p=\sum\limits_{i=1}^{r}\lceil \frac{n_i}{2} \rceil$, which is further conjugate to either $w_{(\lceil\frac{n}{2}\rceil-p, \lceil\frac{n}{2}\rceil+p)}$ or $w_{(\lceil\frac{n}{2}\rceil-p, \lceil\frac{n}{2}\rceil+p]}$, depending on whether $n$ is odd or even. Hence, setting $I=[p+1,n-p]$ we obtain $w_J\sim w_0w_I$.
\subsubsection{Type $B_n/C_n$} Our argument here for type $C_n$ directly adapts to type $B_n$. Here $J\simeq A_{n_1}\times \cdots \times A_{n_r}\times C_l$, with either $l=0$ or  $l\geq 2$. Then $w_J \sim (\prod\limits_{p \text{ copies }}w_{A_1}) w_{B_l}$, with $p=\sum\limits_{i=1}^{r}\lceil \frac{n_i}{2} \rceil$. Hence, setting $I=\{1,3,\cdots, 2p-1\}\coprod \{2p+l+1,\cdots,n\}$ we obtain $w_J\sim w_0w_I$.

\subsubsection{Type $E_6$}\label{type-E6}

\begin{center}\label{ref length}
\scalebox{0.9}{
\begin{tabular}{ |c|c|c|c|c|} 
 \hline
 $(\emptyset, 0)$ & $(A_1, 1)$ & $(A_1^2, 2)$ & $(A_1^3, 3)$ & $(D_4 \text{ or } D_5, 4)$   \\ 
 \hline
 $(\Delta, 4)$ & $(\Delta\setminus \{2\}\sim A_5,3)$ & $(\{3,4,5\}\sim A_3,2)$ & $(\{4\}\sim A_1,1)$ & $(\emptyset ,0)$   \\
 \hline
\end{tabular}}
\end{center}

\subsubsection{Type $E_7$}\label{type-E7} 
\begin{center}
\scalebox{0.9}{

\begin{tabular}{ |c|c|c|c|c|} 
 \hline
 $(\emptyset, 0)$ & $(A_1, 1)$ & $(A_1^2, 2)$ & $(A_1^3 \text{ except } \{2,5,7\}, 3)$ & $(\{2,5,7\}, 3)$   \\ 
 \hline
 $(\Delta, 7)$ & $(\Delta\setminus \{1\}\sim D_6,6)$ & $(\Delta\setminus \{1,6\}\sim D_4\times A_1,5)$ & $(\{1,2,5,7\}\sim A_1^4,4)$ & $(\{2,3,4,5\}\sim D_4,4)$  \\
 \hline
\end{tabular}}
\end{center}
\subsubsection{Type $E_8$}\label{type-E8}
\begin{center}\label{ref length}
\scalebox{0.9}{

\begin{tabular}{ |c|c|c|c|c|c|} 
 \hline
 $(\emptyset, 0)$ & $(A_1, 1)$ & $(A_1^2, 2)$ & $(A_1^3, 3)$ & $(A_1^4, 4)$ & $(D_4,4)$  \\ 
 \hline
 $(\Delta, 8)$ & $(\Delta\setminus \{8\}\sim E_7,7)$ & $(\Delta\setminus \{1,8\}\sim D_6,6)$ & $(\{1,2,5,8\}\sim A_1^4,4)$ & $(\{1,4,6,8\}\sim A_1^4,4)$ & $(D_4,4)$  \\
 \hline
\end{tabular}}
\end{center}
\subsubsection{Type $F_4$}\label{type-F4}
\begin{center}
\scalebox{0.9}{

\begin{tabular}{ |c|c|c|c|c|} 
 \hline
 $(\emptyset, 0)$ & $(\{1\} \text{ or } \{2\}, 1)$ & $(\{3\} \text{ or } \{4\}, 1)$ & $(A_1^2, 2)$ & $(\{2,3\}\sim C_2, 2)$   \\ 
 \hline
 $(\Delta, 4)$ & $(\Delta\setminus \{1\}\sim C_3,3)$ & $(\Delta\setminus \{4\}\sim C_3,3)$ & $(A_1^2,2)$ & $(\{2,3\}\sim C_2, 2)$  \\
 \hline
\end{tabular}}
\end{center}

\subsubsection{Type $G_2$}\label{type-G2}

\begin{center}
\scalebox{0.9}{

\begin{tabular}{ |c|c|c|} 
 \hline
 $(\emptyset, 0)$ & $(\{1\}, 1)$ & $(\{2\}, 1)$   \\ 
 \hline
 $(\Delta, 2)$ & $(\{1\},1)$ & $(\{2\},1)$  \\
 \hline
\end{tabular}}
\end{center}
\section{Proof of the main theorem}\label{sec-proof}
We first discuss the case of Iwahori level structure. The proof of this part is different in spirit from the argument in \cite{HY21}, whereas the proof of the remaining parts is a direct adaptation of the proof scheme in loc. sit.

Basically, the restriction $\mu^\diamond \geq \nu(b)+2\rho^\vee$ enters in the proof of \cite{HY21}[Theorem 6.1] because the dimension for the affine Deligne-Lusztig variety associated to $w:=x t^{\mu}w_0 x^{-1}$ - the specific element chosen in loc. sit. during the proof - can be asserted to be equal to $d_w(b)$ by \cite{He21} only under that stated restriction. Note that it is then shown for this element, we have $d_w(b)=d_{\text{Adm}(\mu)}(b)$ - by which one concludes that $\dim X(\mu,b) = d_{\text{Adm}(\mu)}(b)$. In other words, for this particular choice of $w$, the associated affine Deligne-Lusztig variety has a known dimension by design, and also makes it to the top dimensional component of $X(\mu,b)$.

\begin{proof}[Proof of \cref{main-thm}(i)]
Instead of working with the element $w$ as above, let us consider the element $w':=w_0 t^{\mu-\text{wt}(w_0,1)}$. We claim that

(a) $w' \in \text{Adm}(\mu)$.

Indeed, proceeding as in the proof of \cite{Sad23}[Proposition 4.4] and utilizing the decomposition of $w_0$ in section 5 of loc. sit., we have $$t^{\mu - \text{wt}(w_0,1)} \leq t^{\mu}w_0, \text{~and hence~} w' \leq w_0t^{\mu}w_0.$$ 

This last inequality needs depth at least $2$ to ensure that $\mu - \text{wt}(w_0,1)$, as well the coweights appearing in the intermediate steps (arising from application of the proof technique in loc. sit.), are dominant so that the proof is valid here. Alternatively, we may apply \cite{Sch24}[Proposition 4.12]; either way, we have $w' \in \text{Adm}(\mu)$. 

Now, we claim  that 

(b) $d_{w'}(b) = d_{\text{Adm}(\mu)}(b)$. 

To that end, let us compute 
\begin{align*}
     d_{w'}(b) &=\frac{1}{2}\{\ell(w')+\ell(w_0) - \text{def}(b) - \langle 2\rho, \nu([b]) \rangle \} \\ &= \frac{1}{2}\{\ell(w_0)+\langle 2\rho, \mu - \text{wt}(w_0,1) \rangle +\ell(w_0)-\text{def}(b) - \langle 2\rho, \nu([b]) \rangle\} \\ &= \frac{1}{2}\{\langle 2\rho, \mu - \nu(b) \rangle -\text{def}(b)\} + \{\ell(w_0) - \langle \rho, \text{wt}(w_0,1) \rangle\}.
\end{align*}

Note that \cref{wt-d}[2] gives $\langle 2\rho, \text{wt}(w_0,1) \rangle = \ell(w_0) + d(w_0,1)$, but the explicit decomposition of $w_0$ shows that $d(w_0,1)=\ell_R(w_0)$. Combining, we have $\ell(w_0) - \langle \rho, \text{wt}(w_0,1) \rangle = \ell(w_0) - \frac{1}{2}(\ell(w_0) + \ell_R(w_0)) = \frac{1}{2}\{\ell(w_0) - \ell_R(w_0)\}$.

Finally, we have that

(c) $\text{dim} X_{w'}(b) = d_{w'}(b)$, whenever $X_{w'}(b) \neq \emptyset$.

By \cite[Theorem 1.2]{MV21} every element in the antidominant chamber is cordial. Therefore $w'$ is cordial, and then above claim follows from \cite[Corollary 3.17]{MV21}.

Finally, note further that the basic element in $B(\bG)$ satisfying $\kappa(b)=\kappa(w')$ is indeed the minimal element of $B(\bG)_{w'}$ by \cite{GHN16}[Theorem B]. Hence, $B(\bG)_{w'}=\{[b]: [b] \leq [b_{w'}]\}$. By \cite[Theorem 4.2]{He21pi}, we have $\nu([b_{w'}]) = \mu - \text{wt}(w_0,1)$. This in turn enforces the condition $\mu \geq \nu([b]) +\text{wt}(w_0,1)$ in order to ensure that $X_{w'}(b)\neq \emptyset$. We are done. 
\end{proof}

\begin{proof}[Proof of \cref{main-thm}(ii),(iii)]
Note that $\dim X(\mu,b)_J$ is bounded above by RHS of \cref{dim-formula} by \cref{d-adm} and \cref{min}. 

Let us now consider the element $w:=xt^{\mu-\text{wt}(x,xw_0)}w_0x^{-1}$, where $x\in ~^JW$ is as found in \cref{exceptional}. By \cref{left-part}, $w\in ~^J\wtd{W}$ and by \cite{Sch24}[Proposition 4.12], $w\in \Adm(\mu)$. Since $$\mu-\text{wt}(x,xw_0)-\nu(b)\geq \mu-\text{wt}(w_0,1)-\nu(b)\geq 2\rho^\vee,$$ we have $\dim X_w(b)=d_w(b)$ by \cite{He21pi}[Theorem 1.1]. By a straightforward computation, $d_w(b)=\<\rho, \underline \mu-\nu(b)\>-\frac{1}{2}\text{def}(b)+\frac{1}{2}\{\ell(w_0)-d(x,xw_0)\}$. We are now done by \cref{min}.

\end{proof} 
\begin{remark} Since we may use the explicit construction for $J\subset S$ in type $A_n, B_n, C_n$ - where we have $\text{wt}(x,xw_0)=0$ for the element $x$ found in \cref{sec-abc} - we may modify the proof of \cref{main-thm}[(ii)] accordingly so that it holds under the assumption $\mu\geq \nu(b)+2\rho^\vee$. 

One could ask if we could emulate the proof technique of \cref{main-thm}[(i)] for arbitrary parahoric levels and weaken the gap $\mu-\nu(b)$ further. Here one obstacle is that there may not exist any cordial element $w\in ~^J\Adm(\mu)$ for which $\dim X_w(b)=d_{^{J\Adm(\mu)}}(b)$. 

\end{remark}

\printbibliography[heading=bibintoc]

\end{document}